\documentclass[reqno]{amsart}

\usepackage{amsmath,amssymb,amscd,accents} \usepackage{graphicx}
\DeclareGraphicsExtensions{.eps} \usepackage{mathrsfs}
\usepackage[mathcal]{eucal} \usepackage{esint} \usepackage{color}
\RequirePackage[colorlinks=true, breaklinks=true, urlcolor= black, linkcolor= black, citecolor= black, bookmarksopen=true,linktocpage=true,plainpages=false,pdfpagelabels,unicode]{hyperref}

\long\def\change#1{{\color{blue} #1}}

\newtheorem{Thm}{Theorem}{\bfseries}{\itshape}
\newtheorem*{Thm*}{Theorem}{\bfseries}{\itshape}
\newtheorem{Cor}{Corollary}{\bfseries}{\itshape}
\newtheorem{Prop}[Cor]{Proposition}{\bfseries}{\itshape}
\newtheorem{Lem}[Cor]{Lemma}{\bfseries}{\itshape}
\newtheorem*{Lem*}{Lemma}{\bfseries}{\itshape}
{\bfseries}{\itshape}
{\bfseries}{\itshape}
{\bfseries}{\rmfamily}
{\scshape}{\rmfamily}
\newtheorem{Rem}[Cor]{Remark}{\scshape}{\rmfamily}
{\scshape}{\rmfamily}
\newtheorem{Claim}{Claim}{\bfseries}{\itshape}

\renewcommand\ge{\geqslant} \renewcommand\le{\leqslant}
\let\tildeaccent=\~ \let\hataccent=\^
\renewcommand\~[1]{\widetilde{#1}}

\def\<{\left<} \def\>{\right>} \def\({\left(} \def\){\right)}

\let\polishL=l \def\Zoladek.{\.Zol\c adek}

\def\codim{\operatorname{codim}}

 \def\ord{\operatorname{ord}}
 
\def\etc.{\emph{etc}.}

\def\:{\colon} \def\R{{\mathbb R}} \def\C{{\mathbb C}} \def\Z{{\mathbb
    Z}}  \def\Q{{\mathbb Q}} \def\P{{\mathbb P}}

\def\A{{\mathbb A}} 

\let\PolishL=\L 
\def\L{{\mathbb L}}

 \def\e{\varepsilon}

 \def\Lojas.{\PolishL ojasiewicz} \def\cN{{\mathcal
    N}}

 \def\cO{{\mathcal O}}

 \def\mult{\operatorname{mult}}

\def\OO{\mathcal{O}} \def\mm{\mathfrak{m}}
  \def\QQ{\mathbb{Q}}
  \def\FF{\mathbb{F}}
 \def\PP{\mathbb{P}}

\begin{document}

\title[Sharp bounds for the number of rational points]{Sharp bounds for the number of rational points on algebraic
  curves and dimension growth, over all global fields}

\author[Binyamini]{Gal Binyamini} \address{Weizmann Institute of
  Science, Rehovot, Israel} \email{gal.binyamini@weizmann.ac.il}

\author[Cluckers]{Raf Cluckers} \address{Univ.~Lille, CNRS, UMR 8524 -
  Laboratoire Paul Painlev\'e, F-59000 Lille, France, and KU Leuven,
  Department of Mathematics, B-3001 Leu\-ven, Bel\-gium}
\email{Raf.Cluckers@univ-lille.fr}
\urladdr{http://rcluckers.perso.math.cnrs.fr/}

\author[Kato]{Fumiharu Kato} \address{Department of Mathematics, Tokyo
  Institute of Technology, 2-12-1 Ookayama, Meguro, Tokyo 152-8551,
  Japan} \email{katobungen@gmail.com}

\thanks{ The authors would like to thank Wouter Castryck, Per
  Salberger and Floris Vermeulen for interesting discussions on the
  topics of the paper and the referees for careful comments. G.B. was partially supported by funding from the European
  Research Council (ERC) under the European Union's Horizon 2020
  research and innovation programme (grant agreement No 802107) and by
  the Shimon and Golde Picker - Weizmann Annual Grant. R.C.~was
  partially supported by KU Leuven IF C16/23/010 and the Labex CEMPI
  (ANR-11-LABX-0007-01).}


\begin{abstract}
  Let $C\subset{\mathbb P}_K^2$ be an algebraic curve over a number field $K$,
  and denote by $d_K$ the degree of $K$ over ${\mathbb Q}$. We prove that the
  number of $K$-rational points of height at most $H$ in $C$ is
  bounded by $c d^{2}H^{2d_K/d}(\log H)^\kappa$ where
  $c,\kappa$ are absolute constants. We also prove analogous results
  for global fields in positive characteristic, and, for higher
  dimensional varieties.

  The quadratic dependence on $d$ in the bound as well as the exponent of $H$
  are optimal; the novel aspect is the quadratic dependence on $d$ which answers a question raised by
  Salberger. We derive new results on Heath-Brown's and Serre's dimension
  growth conjecture for global fields, which generalize in particular the results by the
  first two authors and Novikov from the case $K={\mathbb Q}$. The proofs
  however are of a completely different nature, replacing the real
  analytic approach previously used by the $p$-adic determinant
  method. The optimal dependence on $d$ is achieved by a key
  improvement in the treatment of high multiplicity points on mod $p$
  reductions of algebraic curves.
\end{abstract}

\maketitle

\section{Introduction and main results}

When bounding the number of rational points of height at most $H$ on
irreducible curves or higher dimensional varieties of degree $d$
inside $\P^n$ and $\A^n$, the optimal dependence on $H$ is well
understood uniformly throughout such varieties, essentially given by
an appropriate power of $H$ (see e.g. Theorem~\ref{thm:Pila:highdim}
below, quoted from \cite{Pila-ast-1995}). We concentrate on the
dependence on the degree $d$, and obtain 
quadratic
dependence, for both the situation of curves and of hypersurfaces (see
Theorems~\ref{thm:main},~\ref{thm:Pila:highdim;d2},~\ref{thm:main:dg},~\ref{thm:main:dg:aff} and~\ref{thm:main:Pila:d2K}). Our methods are rather elementary, and
work over all global fields $K$ and not just over $\Q$. Previously,
polynomial dependence in $d$ was known in these situations by
\cite{CCDN-dgc,CDHNV-dgc,Pared-Sas,Vermeulen:p}, and, perhaps
surprizingly, this serves as an important ingredient to us. In recent
work by Binyamini, Cluckers and Novikov \cite{BCN-Sal}, quadratic
dependence in $d$ in the curve case with $K=\Q$ was obtained using
one-dimensional real analytic methods, in particular, a criterion by
P\'olya. Here we extend that result to all global fields and to any dimension using a
different approach based on the $p$-adic determinant method. For a
sketch of the argument see Section~\ref{sec:proof-sketch}.

\subsection{Curves}

Around 2002, Salberger raised the question of whether for any $d$, any
irreducible algebraic curve $C\subset \P^n_\Q$ of degree $d$ and any
$H>1$, $\e>0$ one has
\begin{equation}\label{eq:SalQ}
  \# C(\Q,H) \le c d^{2+\varepsilon}H^{\frac{2}{d}+\varepsilon},
\end{equation}
for some $c=c(n,\e)$, where $C(\Q,H)$ denotes the set of rational
points on $C$ with height at most $H$, see \cite{Salberger-preprint},
\cite[below Theorem 0.12]{Salberger-dgc}. This question refines
Heath-Brown's projective variant \cite[Theorem 3]{Heath-Brown-Ann}
of Bombieri-Pila's (affine) results \cite[Theorem 5]{bombieri-pila},
\cite[Main Theorem]{Pila-density-curve} and was motivated by the
Dimension Growth Conjecture, which now is a theorem 
\cite[Theorems 0.1 and
0.3]{Salberger-dgc} (except for the uniform degree $3$ case). 

In this paper, we answer this question about the bounds
(\ref{eq:SalQ}) 
more generally over any global field $K$ (of any characteristic), as
follows.


Let $K$ be a global field. Suppose that $K$ is of degree $d_K$ over
$\Q$ or over $\mathbb{F}_q(t)$ and in the latter case that $K$ is
separable over $\mathbb{F}_q(t)$ and that $K$ has $\mathbb{F}_q$ as
field of constants (as in Section 2.1 of \cite{Pared-Sas}), with $q$ a
power of a prime number. Finally let $C(K,H)$ denote the set of
$K$-rational points on $C$ with absolute projective multiplicative
height at most $H$, see~(\ref{eq:abs:h}) below.

\begin{Thm}\label{thm:main}
  Let $K$ and $d_K$ be as above, and let $n>1$ be given. There is a
  constant $c=c(K,n)$ and an absolute constant $\kappa$ such that for
  any $d>0$, any irreducible algebraic curve $C$ of degree $d$ in
  $\P^n_K$ and any $H>2$ one has
  \begin{equation}\label{eq:Sal:log}
    \# C(K,H) \le c d^{2}H^{\frac{2d_K}{d}}(\log H)^\kappa,
  \end{equation}
Furthermore, one may take $\kappa=12$.
\end{Thm}

The quadratic dependence in the upper bound from (\ref{eq:Sal:log}) is
optimal by \cite[Section 6]{CCDN-dgc}, from which one easily finds a
constant $c'=c'(K)>0$, arbitrarily large values $d$, $H$ and
irreducible curves $C$ in $\P^2_\Q$ of degree $d$ which witness the
lower bound 
\begin{equation}\label{eq:lowerb}
  c' d^{2}H^{\frac{2d_K}{d}}\le \# C(K,H).
\end{equation}

In \cite{BCN-Sal}, the first two authors together with Novikov
recently obtained Theorem~\ref{thm:main} for $K=\Q$, using real
analytic methods which do not seem to extend to the non-real case, and
neither do they extend to the higher dimensional case that we treat in
the next section.
We will give a variant of Theorem \ref{thm:main} for affine curves in Corollary~\ref{cor:AK}.

\subsection{Hypersurfaces}\label{sec:highdimPila}

In higher dimensions, recall the following general bounds by Pila.
Below $X(\Z,H)$ denotes the set of integral points on $X$ with height
at most $H$.

\begin{Thm}[Theorem A of \cite{Pila-ast-1995}]\label{thm:Pila:highdim}
  Let $X\subset \A_\Q^n$ be irreducible, of dimension $m$ and of
  degree $d>0$ for some $n>1$, and let $H> 2$. For any $\e>0$ there is
  a constant $c = c(d,n,\e)$ such that
$$
\# X(\Z,H)\le c\cdot H^{m-1+1/d+\e}.
$$
\end{Thm}

When $X$ from Theorem~\ref{thm:Pila:highdim} is furthermore a
hypersurface, then \cite[Proposition 2.6]{CDHNV-dgc} gives the more
precise bounds
\begin{equation}\label{eq:de1:d}
  \# X(\Z,H)\le  c d^e H^{n-2+1/d}\log H,
\end{equation}
for some $c$ and $e$ depending only on $n$. Again by \cite[Section
6]{CCDN-dgc} it follows that a quadratic dependence on $d$ (that is,
with $e=2$) in the upper bounds (\ref{eq:de1:d}) is best possible. We
obtain this at the cost of extra $\log H$ factors, as follows.
\begin{Thm}\label{thm:Pila:highdim;d2}
  Let $X\subset \A_\Q^n$ be an irreducible hypersurface of degree
  $d>0$ for some $n>1$, and let $H> 2$. Then there are constants
  $c = c(n)$ and $\kappa=\kappa(n)$ such that
$$
\# X(\Z,H)\le cd^2 H^{n-2+1/d}(\log H)^\kappa .
$$
\end{Thm}

We will also give a version for global fields $K$ of Theorem
\ref{thm:Pila:highdim;d2}, see Theorem~\ref{thm:main:Pila:d2K} below,
where we will use a generalization $X(\cO_K,H)$ of $X(\Z,H)$ with
$\cO_K$ the ring of integers of
$K$, 
see (\ref{eq:cOK}) of Section~\ref{sec:prel} (following
\cite{Pared-Sas}, \cite{Bomb-Gubl}).

\subsection{Dimension growth}\label{sec:dgcq}

For dimension growth results, we again get quadratic
dependence on $d$, and, we improve the degree $3$ and $4$ cases by
removing $H^\e$ and replacing it with a power of $\log H$, for all
global fields $K$. Furthermore, the above results for curves give a
simplified approach to dimension growth results for all $K$ and all
degrees $d>2$, compared to the dimension growth results from \cite{Pared-Sas,Salberger-dgc,Vermeulen:p}. This simplification is mentioned as the motivation for Salberger's question
from \cite[below Theorem 0.12]{Salberger-dgc}. 
Perhaps the most surprising aspect is the quadratic
dependence on $d$ for dimension growth (Theorems~\ref{thm:main:dg} and \ref{thm:main:dg:aff}), as well as in Theorem~\ref{thm:Pila:highdim;d2}, since this is new for any $K$; the other mentioned new aspects for dimension growth (namely, the cases of degrees $3$ and $4$, and, the simplified approach) follow similarly from our results for curves (Theorem \ref{thm:main} and Corollary \ref{cor:AK}) as for the case $K=\Q$ in \cite{BCN-Sal}.

\begin{Thm}[Projective dimension growth]\label{thm:main:dg}
  Let $K$ and $d_K$ be as above, and let $n>1$ be given. There are
  constants $c=c(K,n)$ and $\kappa=\kappa(n)$ such that for any
  $d\ge 4$, any irreducible hypersurface $X$ of degree $d$ in $\P^n_K$
  and any $H>2$ one has
  \begin{equation}\label{eq:dg}
    \# X(K,H) \le c d^{2}H^{d_K\cdot\dim X}(\log H)^\kappa.
  \end{equation}
  Furthermore, for any irreducible hypersurface $X$ of degree $3$ in
  $\P^n_K$, one has
  \begin{equation}\label{eq:dg3}
    \# X(K,H) \le c H^{d_K\cdot(\dim X -1 + 2/\sqrt{3} )}(\log H)^\kappa.
  \end{equation}
\end{Thm}

\begin{Thm}[Affine dimension growth]\label{thm:main:dg:aff}
  Let $K$ and $d_K$ be as above, and let $n>2$ be given. There are
  constants $c=c(K,n)$ and $\kappa=\kappa(n)$ such that for any
  $d\ge 4$, any irreducible hypersurface $X$ of degree $d$ given by
  $f=0$ in $\A^n_K$ such that the highest degree part of $f$ is
  geometrically irreducible and any $H>2$, one has
  \begin{equation}\label{eq:dg:aff}
    \# X(\cO_K,H) \le c d^{2}H^{\dim X-1}(\log H)^\kappa.
  \end{equation}
  where $X(\cO_K,H)$ is defined in (\ref{eq:cOK}) of
  Section~\ref{sec:prel} below.

  Furthermore, for any irreducible hypersurface $X$ of degree $3$
  given by $f=0$ in $\A^n_K$ such that the degree $3$ part of $f$ is
  geometrically irreducible, one has
  \begin{equation}\label{eq:dg3:aff}
    \# X(\cO_K,H) \le c H^{\dim X  -2 + 2/\sqrt{3}}(\log H)^\kappa.
  \end{equation}
\end{Thm}

In the above Theorems~\ref{thm:main:dg} and~\ref{thm:main:dg:aff} with
$d=4$ and $d=3$, we have removed $H^\e$ and replaced it with
$(\log H)^\kappa$, compared to \cite[Theorem 1.11]{Pared-Sas}. This
gives an answer to a question by Serre \cite[page 178]{Serre-Mordell}
for all $d\ge 4$, the case of degree $4$ being new for general $K$
(the case $K=\Q$ was obtained in \cite{BCN-Sal}). See Theorem 4.1 of \cite{CDHNV-dgc} for a relaxation of the conditions on $f$, compared to the conditions in Theorem~\ref{thm:main:dg:aff}. Also note that, if one allows a higher power of $d$ than quadratic, the factor $(\log H)^\kappa$ is not needed in Theorems \ref{thm:main:dg} and \ref{thm:main:dg:aff} when $d>4$, see Theorems 4.1 and 4.20 of \cite{CDHNV-dgc}, improving \cite{Pared-Sas}.

In Theorem~\ref{thm:main:dg:aff}, quadratic dependence on
$d$ is optimal by \cite[Section~6]{CCDN-dgc}.
For the projective situation of Theorem~\ref{thm:main:dg}, a lower bound which is linear in $d$ (instead of
quadratic) is obtained in \cite[Section~6]{CCDN-dgc}; this lower bound is improved to an almost quadratic bound in \cite{CGlaz}, namely up to a factor $d^\varepsilon$ when the dimension is allowed to be large.   







\subsection{Sketch of the proof}\label{sec:proof-sketch}

We sketch the main ideas in the proof of Theorem~\ref{thm:main} for
the basic case $K=\Q$, which go through with small technical
modifications in the general case. It is relatively easy to reduce to
the case $d>(\log H)^2$, because otherwise the result already follows
from known results (from \cite{me:c-cells} or \cite{CCDN-dgc}). Note that in this regime, the $H^{2/d}$ term
in~(\ref{eq:Sal:log}) is bounded by a constant, so our goal is to
prove a bound of order $d^2(\log H)^\kappa$. Also suppose for
simplicity that $C\subset\P^2_\Q$.

We choose a prime $p$ of size roughly $(\log H)^4$. Denote by
$C_p\subset\P^2_{\FF_p}$ the modulo $p$ reduction of $C$, fix
a point $P$ in $C_p(\FF_p)$ and write $\mu_P$ for the multiplicity of $P$ on $C_p$. We will treat each of
the points $P$ separately, as there are at most $p^2\sim (\log H)^8$
of them. Our goal is to use the $p$-adic interpolation determinant
method \cite{Heath-Brown-Ann,Salberger-dgc} to interpolate all
rational points in $C(\Q,H)$ reducing to $P$ in $C_p$ by an algebraic
curve of degree $d-1$. This would give us a bound of order $d^2$ for
the number of points lying over $P$ by B\'ezout. 
In the original  $p$-adic determinant method from \cite{Heath-Brown-Ann}, when one considers $d$ as fixed and
$H\gg1$, it is necessary to work with points $P$ satisfying
$\mu_P=1$. If one allows singular points as in \cite{Salberger-dgc}, the determinant method
generally only gives a bound of order $H^{2\mu_P/d}$ rather than the
requisite $H^{2/d}$. A key observation is that in our special regime
$d>(\log H)^2$ the condition $\mu_P=1$ can be significantly relaxed
--- in fact assuming just $\mu_P<d/\log H$ we already have
\begin{equation}
  H^{2\mu_P/d} \sim 1 \sim H^{2/d} .
\end{equation}
In other words, we can treat, with the standard $p$-adic determinant
method, all points $P\in C_p(\FF_p)$ except those having
$\mu_P\ge d/\log H$. We refer to these latter points as having ``high
multiplicity'', and to the others as having ``low multiplicity''. We
show in Proposition~\ref{lem:alpha} that all points of high multiplicity lie
in an algebraic curve in $\P^2_{\FF_p}$ of degree $N\log H$ for some
universal constant $N$; the proof of Proposition~\ref{lem:alpha} relies on some basic intersection theory.

We now repeat this argument with roughly $(\log H)^4$ different primes
$p_i$. Any point of $C(\Q,H)$ that has low multiplicity modulo one of
these primes is handled by the $p_i$-adic interpolation determinant
method as above. Finally, the remaining points lie on a curve of
degree $N\log H$ modulo each $p_i$. This means that a polynomial
interpolation determinant of degree $N\log H$ for these points must
vanish modulo each $p_i$. Comparing this with the height bound for the
determinant implies that the determinant vanishes identically, i.e.,
that all remaining points are interpolated by one additional algebraic
curve of degree at most $N\log H<d$. Another application of B\'ezout's theorem
concludes the proof of Theorem~\ref{thm:main}.

\section{High multiplicity points}

To prove Theorem~\ref{thm:main}, we will use the following Proposition
\ref{lem:alpha} for capturing the high multiplicity points in a low
degree auxiliary curve. A reader who wants to continue with the proof
of Theorem~\ref{thm:main} can skip directly to the next section, and
come back to the proof of Proposition~\ref{lem:alpha} later.

In fact, Proposition~\ref{lem:alpha} will be obtained as a special
case of Proposition~\ref{lem:alpha:highdim:1} (which in turn is a
special case of Proposition~\ref{lem:alpha:highdim}).  It is important
to us that these
propositions~\ref{lem:alpha},~\ref{lem:alpha:highdim:1},~\ref{lem:alpha:highdim}
work over any field $F$, of any characteristic.

We recall the definition of the multiplicity of a point on a
hypersurface here, and the more general definition of multiplicity
will be recalled just before Proposition~\ref{lem:alpha:highdim}.  If
$X$ is a hypersurface in a non-singular variety $Y$ and $P$ a
point in $Y(F)$, then the multiplicity $\mult_P(X)$ of $X$ at $P$ is
defined as follows.  Take local coordinates $x=(x_1,\ldots,x_n)$
around $P$ on $Y$ (in particular, with $P$ having coordinates $0$),
and let $X$ be locally defined by a regular function
$f=f(x_1,\ldots,x_n)=0$.  Write
$$
f(x)=\sum_{k\geq 0}f_k(x),
$$
where $f_k(x)$ is a homogenous polynomial of $x$ of degree $k$, and
where the equality holds in the completed local ring
$\hat \cO_{Y,P}\simeq F[[x]]$.  Then $\mult_P(X)$ is equal to the
smallest $k$ with $f_k\neq 0$.

\begin{Prop}\label{lem:alpha}
  Let $F$ be a field and let $f$ be a homogeneous polynomial in
  $F[x_0,x_1,x_2]$ 
  of degree $D>0$. Let $\Gamma\subset \P^2_F$ be the corresponding
  (potentially non-reduced and reducible) curve of degree $D$. Then,
  for any real number $k\ge 1$, the points in $\Gamma(F)$ with
  multiplicity at least $D/k$ lie in a curve
  of degree no more than $Nk$ for some universal constant $N$.
\end{Prop}

For the proofs of our higher dimensional results like
Theorems~\ref{thm:Pila:highdim;d2},~\ref{thm:main:dg},~\ref{thm:main:dg:aff},
we will use the following generalization of
Proposition~\ref{lem:alpha} to capture the high multiplicity points
on a general hypersurface by a low degree hypersurface.

\begin{Prop}\label{lem:alpha:highdim:1}
  Let $F$ be any field. Let $k\ge 1$ and $n\ge
  2$ 
  be given. Consider a homogeneous polynomial $f$ in $n+1$ variables
  over $F$, of degree $D$. Then the set of points $P $ in $\P^n_F(F)$ on $f=0$ where
  $\mult_P (f) > D/k$ lies in a hypersurface of degree no more than
  $ck^{2^{n-2}}$ for some constant $c=c(n)$ depending only on
  $n$. Here, $\mult_P (f)$ stands for the order of vanishing of $P$ on
  $f$.
\end{Prop}

In fact, Proposition~\ref{lem:alpha:highdim:1} is a direct corollary
of Proposition~\ref{lem:alpha:highdim}, for effective cycles of any
codimension instead of hypersurfaces, that we will prove by induction
on the dimension. Let us first recall the general definition of
multiplicity (a special case of which was already recalled above).

Let $X$ be a variety over a field $F$, and $P\in X$ a closed point.
Consider the local ring $\OO_{X,P}$ at $P$ and the maximal ideal
$\mm\subset\OO_{X,P}$.
Consider the graded ring
$$
A=\bigoplus_{i=0}^{\infty}\mm^i/\mm^{i+1}.
$$
If $P_A(\mm)$ is the Hilbert polynomial of the graded ring $A$, then
$$
\mult_P(X):=(\dim X-1)!\cdot(\textrm{the leading coefficient of
  $P_A(\mm)$})
$$
is the {\em multiplicity of $X$ at $P$}.

Recall that an effective cycle in a non-singular variety $W$ is a
linear combination of irreducible varieties in $W$ with non-negative
integer coefficients.  One can extend the definition of multiplicities
to an effective cycle $X=\sum_i a_iX_i$ in a non-singular variety by
linearity, that is, by
$$
\mult_P(X) :=\sum_i a_i \mult_P(X_i),
$$
where $\mult_P(X_i)$ is defined as zero when $P$ does not lie on
$X_i$.  See for example \cite[Chap.\ 12]{Eisenbud} or
\cite{hartshorne} for more detail about multiplicities and Hilbert
polynomials.  Recall that the degree an effective cycle
$X=\sum_i a_iX_i$ in $\PP^n$ is defined by linearity, that is, by
$$
\deg (X) :=\sum_i a_i \deg (X_i),
$$
where $\deg (X_i)$ is the degree of $X_i$.

\begin{Prop}\label{lem:alpha:highdim}
  Let $F$ be any field. Let $k\ge 1$, $n\ge 2$ and $m\ge 1$ with
  $m<n$ 
  be given.  Let $\Gamma$ be an effective cycle in $\P_F^n$ of pure
  dimension $m$ and of degree $D>1$. Then the set of points $P $ on
  $\Gamma(F)$ where $\mult_P \Gamma > D/k$ lies in a hypersurface of
  degree no more than $ck^{2^{m-1}}$ for some constant $c=c(n)$
  depending only on
  $n$. 
\end{Prop}

Note that the constants $N$ and $c$ from Propositions \ref{lem:alpha},  \ref{lem:alpha:highdim:1}  and \ref{lem:alpha:highdim} can in principle be made explicit from our proofs, but we will not need this.

Let us first recall some basics about multiplicities and intersection
products.

We will use the following lemma in the proof of Proposition
\ref{lem:alpha:highdim}.
\begin{Lem}\label{lem:multipicity3}
  Let $F$ be any field.  Let $X=\sum_ia_iX_i$ and $Y=\sum_jb_jY_j$ be
  pure dimensional effective cycles of $\PP^n_F$, and $P$ a closed
  point in $X\cap Y$.  Suppose that $X$ and $Y$ intersect properly,
  i.e., for any $i,j$ and any irreducible component $Z$ of
  $X_i\cap Y_j$, we have $\codim(Z)=\codim(X_i)+\codim(Y_j)$.  Then we
  have
$$
\mult_P(X\cdot Y)\geq \mult_P(X)\mult_P(Y),
$$
where $X\cdot Y$ is the intersection product on the level of cycles,
which is well-defined since $X$ and $Y$ intersect properly.
\end{Lem}

Note that the intersection product $X\cdot Y$ in Lemma
\ref{lem:multipicity3} is a linear combination of the irreducible
components of the $X_i\cap Y_j$.  For the sake of completeness, we
provide a proof for Lemma~\ref{lem:multipicity3}, based
on 
Lemma~\ref{lem:multiplicity2}.  For pure dimensional cycles $X,Y$ of a
non-singular variety $W$ which intersect properly, and an irreducible
component $Z$ of $X\times_W Y$, we denote by
$$
i(Z,X\cdot Y; W)
$$
the intersection multiplicity of $X$ and $Y$ along $Z$, see
\cite[Chapter 7]{Fulton}.
\begin{Lem}\label{lem:multiplicity2}
  Assume that $F$ is an algebraically closed field.  Let $X$ be a pure
  dimensional effective cycle of $\PP^n_F$, and $P\in X$ a closed point.  Then for a
  general linear subspace $L$ of $\PP^n_F$ such that
  $\dim(L)=\codim(X)$ and $P\in L$, we have
$$
\mult_P(X)=i(P,X\cdot L;\PP^n_F).
$$
\end{Lem}
\begin{proof}[Proof of Lemma~\ref{lem:multiplicity2}]
  We may assume that $X$ is an irreducible closed subvariety of
  $\PP^n_F$.  It is well-known that the dimension of the tangent cone
  at a point of $X$ is equal to $\dim(X)$.  Hence it follows from
  \cite[p.~227, Corollary 12.4]{Fulton} that
  $i(P,X\cdot L;\PP^n_F)=\mult_P(X)\mult_P(L)=\mult_P(X)$ for a
  general choice of $L$.
\end{proof}
\begin{proof}[Proof of Lemma~\ref{lem:multipicity3}]
  We may assume that $F$ is algebraically
  closed. 
  Take a linear subspace $L\subset\PP^n$ of dimension
  $\codim(X)+\codim(Y)$ passing through $P$ such that
  $\mult_P(X\cdot Y)=i(P,(X\cdot Y)\cdot L;\PP^n)$, by
  Lemma~\ref{lem:multiplicity2}.  Since $X\cdot L$ and $Y$ are of
  complementary dimension, we have
  \begin{equation*}
    \begin{split}
      \mult_P(X\cdot Y)&=i(P,(X\cdot Y)\cdot L;\PP^n)\\
                       &=i(P,(X\cdot L)\cdot Y;\PP^n)\\
                       &\geq \mult_P(X\cdot L)\mult_P(Y),
    \end{split}
  \end{equation*}
  where the last inequality follows from \cite[p.~227, Corollary
  12.4]{Fulton}.  Now, again in view of Lemma~\ref{lem:multiplicity2},
  take a linear subspace $L'$ of dimension $\codim(X\cdot L)$ passing
  through $P$ such that
  $\mult_P(X\cdot L)=i(P,(X\cdot L)\cdot L';\PP^n)$.  Then,
  \begin{equation*}
    \begin{split}
      \mult_P(X\cdot L)&=i(P,(X\cdot L)\cdot L';\PP^n)\\
                       &=i(P,X\cdot (L\cdot L');\PP^n)\\
                       &\geq \mult_P(X)\mult_P(L\cdot L')=\mult_P(X),
    \end{split}
  \end{equation*}
  where again the inequality follows from \cite[p.~227, Corollary
  12.4]{Fulton}.  Lemma~\ref{lem:multipicity3} is now proved.
\end{proof}

\begin{proof}[Proof of Proposition~\ref{lem:alpha:highdim}]

  We will proceed by induction on $m$ but first we reason for general
  $m\ge 1$.

  It is enough to prove the proposition under the extra assumption
  that $F$ is algebraically closed, since the general case follows
  from this case.  Indeed, suppose that we know the proposition over
  the algebraic closure of $F$, but not over $F$ itself. Then, any
  interpolation determinant built up with the homogeneous monomials of
  degree $\lfloor ck^{2^{m-1}}\rfloor $ evaluated in some points of
  $\Gamma(F)$ vanishes, and hence, by linear algebra we find a
  hypersurface of degree no more than $ck^{2^{m-1}}$ defined already
  over $F$.

  Let us from now on suppose that $F$ is algebraically closed.  In
  this proof, by `generic' we mean outside some proper Zariski closed
  subset defined over $F$.  By projecting $\Gamma$ to $\P^{m+1}$
  (along a generic projection) there is furthermore no harm in
  assuming that $\Gamma$ has codimension $1$, that is,
  $m=n-1$. Indeed, a generic projection is birational on each
  irreducible component appearing in the effective cycle $\Gamma$, and
  preserves the degree of each irreducible component of
  $\Gamma$. Furthermore, for such a generic projection $\pi$, one has
  that the multiplicity of a point $P$ on $\Gamma$ is less or equal to
  the multiplicity of $\pi(P)$ on $\pi(\Gamma)$ (see
  e.g.~\cite[Prop.\,8.3]{Rydh}, \cite[Ex.\, 4.9\, Part I]{hartshorne},
  and Zariski's multiplicity formula for finite projections
  \cite[p.~297]{Zariski-Samuel}). This shows we can assume that
  $m=n-1$ holds.


  Now write
$$
\Gamma = \sum n_j C_j
$$
with (mutually different) irreducible hypersurfaces $C_j$ in $\P_F^n$,
and integers $n_j>0$.
For each $j$ let $C'_j $ be the effective cycle given by
$$
{\sum_{i=0}^n a_i \frac{d}{dx_i} f_j=0}
$$
if $C_j$ is given by ${f_j=0}$, for some generic $(a_0,\ldots,a_1)$ in
$F^{n+1}$. Note that $C_j'$ is of dimension $m$. Indeed, its defining
polynomial cannot identically vanish since $F$ is algebraically closed
and $C_j$ irreducible (and thus, $f$ cannot be a sum of $p$-th
powers).  Consider the effective cycle of pure dimension $n-2 = m-1$
$$
A = \sum_j n_j C_j \cdot (n_j C'_j + \sum_{\ell \neq j} n_\ell C_\ell)
$$
where the product is the intersection product (note that $C'_j$
intersects $C_j$ properly, and, that $C_j$ intersects $C_\ell$
properly when
$\ell\not=j$). 


B\'ezout's Theorem implies that
\begin{equation}\label{A:D2}
  \deg A < D^2.
\end{equation}
We may clearly assume that $D/k\ge 2$ since otherwise there is nothing
to prove as one can simply take $\Gamma$ itself.
\begin{Claim}
  For every point $P $ on $\Gamma(F)$ with $\mult_P \Gamma > D/k$, one
  has
  \begin{equation}\label{eq:(*)}
    \mult_P  A > D^2/8k^2, 
  \end{equation}
  unless $P $ is smooth in some $C_j$ and $n_j >D/(2k)$.
\end{Claim}

Let us prove Claim 1. Suppose first
that $P $ is singular on some $C_{j_0}$ with
$$
n_{j_0} \mult_P (C_{j_0}) > D/(2k).
$$
Then we have $\mult_P (C'_{j_0}) \ge \mult_P (C_{j_0})-1\ge 1$ and
thus, by Lemma~\ref{lem:multipicity3}, 
\begin{multline*}
\mult_P A = \mult_P \left( \sum_{j,\ P \in C_j} n_j C_j \cdot (n_jC_j' + \sum_{\ell\not=j,\ P \in C_\ell} n_\ell
  C_\ell)\right) \\
  \ge \mult_P ( n_{j_0}C_{j_0}\cdot n_{j_0}C_{j_0}') \ge n_{j_0}^2 \left( (\mult_P (C_{j_0}))^2 - \mult_P (C_{j_0}) \right) > D^2/8k^2,
\end{multline*}
and (\ref{eq:(*)}) follows. Finally suppose that
$n_j \mult_P (C_j) \le D/(2k)$ for each $j$ such that $P $ belongs to
$C_j$. Note that $P $ lies in more than one component, by our
assumption $\mult_P \Gamma > D/k$.  We calculate, using Lemma
\ref{lem:silly-arithmetic},
\begin{multline*}
  \mult_P A \ge \mult_P \left( \sum_{j,\ P \in C_j} n_j C_j \cdot (\sum_{\ell\not=j,\ P \in C_\ell} n_\ell C_\ell)\right) \\
  \ge \frac{1}{2} (\sum_{j,\ P \in C_j} n_j (\mult_P C_j))^2 \ge \frac{1}{2} (\mult_P \Gamma)^2 > D^2/2k^2.
\end{multline*}
This proves the claim.

We can cover all these points where (\ref{eq:(*)}) fails by
considering the hypersurface $\sum_j C_j$ where the sum is taken over
all $j$ that satisfy $n_j>D/(2k)$; note that its degree is no more
than $ck$ for some constant
$c$. 

Now, for the points where (\ref{eq:(*)}) is satisfied, the claim
follows by induction on $m$, as follows. If $m=1$, then by
(\ref{A:D2}) there are no more than $\sim k^2$ many points $P$ on
$A(F)$ satisfying (\ref{eq:(*)}), which can clearly be interpolated by
a curve of degree $\sim k$.

If $m>1$, we apply the induction hypothesis to $A$, with $8k^2$
instead of $k$ and $D^2$ instead of $D$.
\end{proof}

\begin{Lem}\label{lem:silly-arithmetic}
  Let $x_1,\ldots,x_n$ be non-negative integers and assume
  $2x_j<x_1+\cdots+x_n$ for all j. Then
  \begin{equation*}
    (x_1+\cdots+x_n)^2 < 2 \sum_{j\neq \ell } x_j x_\ell.
  \end{equation*}
\end{Lem}
\begin{proof}
  It is enough to show
  $x_1^2+\cdots+x_n^2 < \sum_{j\neq \ell } x_j x_\ell $, and this is
  clear since
  \begin{multline*}
    x_1^2+\cdots+x_n^2 < \\
    x_1(x_2+\cdots+x_n)+x_2(x_1+x_3+\cdots+x_n)+\cdots+x_n(x_1+\cdots+x_{n-1}) = \\
    \sum_{j\neq \ell } x_j x_\ell.
  \end{multline*}
\end{proof}

\section{Proof of Theorem~\ref{thm:main} for $K=\Q$}



\subsection*{Step 1. Basic reductions}\label{step:1}

We 
reduce to the case with $n=2$
and with furthermore
\begin{equation}\label{eq:logHd}
  1\ll (\log H)^2 < d <  H^{3/2}.  
\end{equation}


By \cite[Proposition 4.3.2]{CCDN-dgc} for $m=1$ we find a degree $d$ curve $C'$ in $\P^2_\Q$ which is birationally equivalent to $C$ and which satisfies for all $H>2$
\begin{equation}\label{eq:CC'}
\#C(\Q,H)\le \#C'(\Q,c_nd^{e_n}H) + d^2
\end{equation}
for some constants $c_n$ and $e_n$ depending only on $n$. Hence, it is sufficient to prove Theorem~\ref{thm:main} for $C'$ since obviously we may assume that $d<H^{(n+1)/2}$ from
\begin{equation}
  \# C(\Q,H) \le \# \P^n_\Q(\Q,H) \le c' H^{n+1}
\end{equation}
for some $c'$.
By \cite[Theorem 2]{CCDN-dgc}, one has
\begin{equation}\label{eq:d4}
  \# C'(\Q,H) \le c d^{4}H^{\frac{2}{d}},
\end{equation}
for some absolute constant $c$, and, clearly one has
\begin{equation}\label{eq:H3}
\# C'(\Q,H) \le \# \P^2_\Q(\Q,H) \le c' H^{3},
\end{equation}
for some constant $c'$.
If the negation of (\ref{eq:logHd}) holds, then
the bounds (\ref{eq:Sal:log}) of Theorem
\ref{thm:main} for $K=\Q$ clearly follow for $C'$ from the above inequalities (\ref{eq:d4}) and (\ref{eq:H3}), with $\kappa= 4$.
To prove Theorem~\ref{thm:main} with $K=\Q$ it is thus sufficient to
prove the following Proposition.

\begin{Prop}\label{prop:g}
  With any $C$, $d$ and $H$ as in Theorem~\ref{thm:main}, $\kappa=12$,
  and under the extra assumptions $K=\Q$, $n=2$ and the inequalities (\ref{eq:logHd}),
  there exist no more than $c(\log H)^\kappa$ many polynomials $g_i$
  of degree $d-1$ whose joint zero locus contains $C(\Q,H)$, for some
  absolute constant $c$.
\end{Prop}

Note that Theorem~\ref{thm:main} with $K=\Q$ 
clearly follows from Proposition~\ref{prop:g} and the above basic reduction to $n=2$
and to  (\ref{eq:logHd}), by B\'ezout's Theorem. This finishes the
basic reductions from Step 1. 




\subsection*{Step 2. Points of low multiplicity after reducing modulo
  prime numbers}

Let $C$, $d$ and $H$ be as in Theorem~\ref{thm:main} with $K=\Q$, and
assume $n=2$ and the inequalities from (\ref{eq:logHd}).

Take any prime number $p$ with $p > \log
H$, 
and let $C_p$ be the scheme-theoretic reduction of $C$ modulo $p$
inside $\P^2_{\mathbb{F}_p}$. Note that $C_p$ is a curve, of degree at
most $d$.  Indeed, we may suppose that $C$ is given by a primitive
polynomial with integral coefficients.
Consider a point $P$ on $C_p(\FF_p)$ and write $\mu_P$ for
$\mult_P(C_p)$, the multiplicity of $P$ on $C_p$. Let $\Xi$ be a
collection of points in $C(\Q,H)$ such that each point in $\Xi$ has
some primitive integer tuple of homogeneous coordinates which reduces
to $P$ modulo $p$. Assuming that $\mu_P$ is not too large, we want to
find a curve $g=0$ in $\P^2$ of degree $d-1$ such that $g$ vanishes on
all the points in $\Xi$.  To this end, we form an interpolation
determinant $\Delta$, namely, the determinant of a matrix
$(m_i(\xi_j))_{i,j}$ consisting of all monomials $m_i$ of degree $d-1$
in $x_0,x_1,x_2$ evaluated in some points $\xi_j$ of $\Xi$, or more
precisely, evaluated in a primitive integer tuple
$(\xi_{j1},\xi_{j2},\xi_{j3})$ of homogeneous coordinates of $\xi_j$,
for some $\xi_j$ in $\Xi$. This matrix is of dimension $s \times s$
with
$$
s:={d+1\choose 2}, 
$$
and, each entry is no larger than $H^{d-1}$.  Hence, from
(\ref{eq:logHd}), we find
$$
|\Delta| \ll d^2! H^{d^3} \ll d^{2d^2} H^{d^3}
\ll 
H^{2d^3}.
$$
By Corollary 2.5 of \cite{CCDN-dgc} (which renders \cite[Main Lemma
2.5]{SalbCrelle} independent of
$d$), 
we find that $\Delta$ is divisible by $p^e$ with
\begin{equation}\label{eq:emu}
  e \ge \frac{s^2}{2\mu_P} - as
\end{equation}
for some absolute constant $a$.


Let us assume that
\begin{equation}\label{eq:muP}
  \mu_P  
  < \frac{d}
  {\log H},
\end{equation}
so that 
(\ref{eq:logHd}) and (\ref{eq:emu}) imply
\begin{equation}\label{eq:e:log}
  e 
  \ge a'  \log H\cdot  d^3 
\end{equation}
for some constant $a'>0$.
If $\Delta$ is nonzero, we find 
$$
H^{a' (\log(\log H)) d^3} = (\log H)^{a' (\log H) d^3} < p^{a' (\log H) d^3} \le p^{e} \le |\Delta| \ll H^{2d^3},
$$

which clearly gives a contradiction
with 
(\ref{eq:logHd}).

Concluding, assuming $n=2$ and (\ref{eq:logHd}), and with $p$, $P$,
$\Xi$ 
and $\mu_P$ as in the beginning of Step 2,
if 
$\mu_P < d/ 
\log H$, then there is a homogeneous polynomial $g$ of degree $d-1$
vanishing on all points in $\Xi$.





\subsection*{Step 3. Proof of Proposition~\ref{prop:g}} 

Using the above material, we can now prove Proposition~\ref{prop:g}.


\begin{proof}[Proof of Proposition~\ref{prop:g}] 




  Put $d':= \lfloor N \log H
  \rfloor $ with constant $N$ from Proposition~\ref{lem:alpha}.  
  Consider all primes $p_i$ between $\log H$ and $M(\log H)^4$ for
  some large constant $M$. 

  Let $\Xi_s$ be the set of points on $C(\Q,H)$ whose reduction modulo
  $p_i$ has multiplicity at least $d/\log H$ for each of our primes
  $p_i$. Note that the reduction of $C$ modulo $p_i$ is a curve,
  since we may suppose that $C$ is
  given by a primitive homogeneous polynomial $f$ over $\Z$.  Any
  interpolation determinant $\Delta_{\Xi_s,d'}$ with monomials of
  degree $d'$ and points from $\Xi_s$ is divisible by all the chosen
  primes $p_i$, by Proposition~\ref{lem:alpha}, and where an interpolation determinant is as explained in Step 2.  Hence, if
  $\Delta_{\Xi_s,d'}$ is nonzero we have
  \begin{equation}\label{eq:d'}
    H^{c'M(\log H)^{3}} = e^{c'M(\log H)^4 } < \prod_{\log H < p_i < M(\log H)^4 } p_i \le |\Delta_{\Xi_s,d'}|,
  \end{equation}
  for some constant $c'$, where we have used that there are positive
  constants $c_1$ and $c_2$ with
  \begin{equation}\label{eq:T}
    c_1T< \sum_{p<T} \log p < c_2T,
  \end{equation}
  for large $T$, where the sum is over primes smaller than $T$, and
  where (\ref{eq:T}) follows from the prime number theorem.

  On the other hand, using (\ref{eq:logHd}) we find that
  \begin{equation}\label{eq:d''}
    |\Delta_{\Xi_s,d'}| \le  d'{}^2! H^{d'{}^3} \ll d'{}^{2d'{}^2} H^{d'{}^3}     < H^{2d'{}^3} < H^{c''\log H ^3} ,
  \end{equation}
for some constant $c''$.
  By (\ref{eq:d'}), (\ref{eq:d''}), and by taking $M$ large enough,
  $\Delta_{\Xi_s,d'}$ must vanish as soon $H$ is large enough (which
  holds by (\ref{eq:logHd})). 
  Hence, there is a polynomial $g_0$ of degree $d'$ that vanishes on
  $\Xi_s$. Observe that $d'<d$ by (\ref{eq:logHd}).

  The complement of $\Xi_s$ in $C(\Q,H)$ is treated using Step 2,
  as follows. 
  For each prime $p_i$ with $\log H < p_i < M(\log H)^4 $, and each
  point $P$ on $C_{p_i}(\mathbb{F}_{p_i})$ of multiplicity at most
  $d/\log H$, there is a polynomial of degree $d-1$ that vanishes on
  the subset of $C(\Q,H)$ consisting of points which reduce to $P$
  modulo $p_i$ (see Step 2 above).
  There clearly are no more than $c(\log H)^4 /\log (\log H)$ many
  such primes $p_i$, for some constant $c$ (by the prime number
  theorem), and for each such $p_i$ there are no more than
  $c'(\log H)^{8}$ many points on $C_{p_i}(\mathbb{F}_{p_i})$ for some $c'$.  Hence,
  in total there are no more than
$$
c(\log H)^{12}/\log ( \log H)
$$
many polynomials, as desired for Proposition~\ref{prop:g}, for some
constant $c$.
Proposition~\ref{prop:g} and thus also Theorem~\ref{thm:main} with
$K=\Q$ 
now
follows.  
\end{proof}

For the sake of completeness, we formulate a variant
of Theorem \ref{thm:main} for affine curves.
See Corollary~\ref{cor:AK} for a variant over any global field $K$.

\begin{Cor}[Affine curves]\label{cor:AQ}
  There is a constant $c=c(n)$ and an absolute constant $\kappa$ such
  that for any $d>1$, any irreducible algebraic curve $C$ of degree
  $d>0$ in $\A^n_\Q$ with $n>1$ and any $H>2$ one has
  \begin{equation}\label{eq:Sal:log:AO}
    \# C(\Z,H) \le c d^{2}H^{\frac{1}{d}}(\log H)^\kappa.
  \end{equation}
\end{Cor}
\begin{proof}
  One can make a similar reduction as in Step 1, using \cite[Theorem 3]{CCDN-dgc} instead of \cite[Theorem 2]{CCDN-dgc}. Under these
  so-obtained assumptions, the corollary follows immediately from
  Theorem~\ref{thm:main} (by working with the closure of $C$ in
  $\PP^2$) since $H^{2/d}$ is bounded by a constant when $d>\log H$.
\end{proof}
The above proofs in this section have given a new approach to \cite[Theorems~1 and 2]{BCN-Sal}, as explained in the introduction. The upshot is that these new proofs generalize directly to all global fields, and, to higher dimensions, as we explain in the next sections. Another upshot is that the exponent $\kappa$ of $\log H$ in our main theorems can be made explicit; we leave this to the reader (except for Theorem \ref{thm:main} where $\kappa=12$), see Remark \ref{rem:highkappa}.



\section{Proof of Theorem~\ref{thm:Pila:highdim;d2}} 

The bulk for proving Theorem~\ref{thm:Pila:highdim;d2} is the
following analogue of Proposition~\ref{prop:g}.

\begin{Prop}\label{prop:g:Pila}
  With $X$, $d$ and $H$ as in Theorem~\ref{thm:Pila:highdim;d2} and
  with the extra assumptions that $n>2$ and
  \begin{equation}\label{eq:logHd:Pila:p}
    1\ll_n (\log H)^N < d <  H 
  \end{equation}
  for some constant $N=N(n)$, there exist no more than
  $c(\log H)^\kappa$ many polynomials $g_i$ of degree $d-1$ whose
  joint zero locus contains $X(\Z,H)$, for some constants $c=c(n)$ and
  $\kappa=\kappa(n)$.
\end{Prop}
\begin{proof}
  Consider the primes $p_i$ between $\log H$ and $M(\log H)^\ell$ for
  some large $M=M(n)$ and $\ell=\ell(n)$. Let $\Xi_s$ be the subset of
  $X(\Z,H)$ consisting of those points whose reduction modulo each of
  these primes $p_i$ has multiplicity at least $d/(\log H)^\alpha$ for
  some large $\alpha=\alpha(n)$. Then any interpolation determinant with
  monomials of degree no more than $d' := (\log H)^C$ on the points of
  $\Xi_s$ is divisible by all these primes $p_i$, for some constant
  $C=C(n)$, by Proposition~\ref{lem:alpha:highdim:1}. By a calculation
  as for the high multiplicity points in the proof of
  Proposition~\ref{prop:g}, one finds that $\Xi_s$ lies inside the
  zero locus of a single polynomial of degree $d'$, and, one has
  $d'<d$ by (\ref{eq:logHd:Pila:p}).

  Let us now consider the complement of $\Xi_s$ inside $X(\Z,H)$, and
  write it as a finite union of subsets $\Xi_P$ for each of our primes
  $p_i$ and each low multiplicity point $P$ on $X_{p_i}$, the
  reduction of $X$ modulo $p_i$. Here, low multiplicity means
  multiplicity at most $d/(\log H)^\alpha$. Let $\Delta$ be an
  interpolation determinant with monomials of degree $<d$ and points
  in $\Xi_P$.  Write $s$ for the number of monomials of degree $<d$ in
  $n$ variables.

  By Corollary 2.5 of
  \cite{CCDN-dgc} 
  we find, by homogenizing, that $\Delta$ is divisible by $p_i^e$ with
  \begin{equation}\label{eq:emu:bis}
    e \ge \left(\frac{(n-1)!}{\mu_P}\right)^{\frac{1}{n-1}}\frac{n-1}{n} s^{1+\frac{1}{n-1}} - as
  \end{equation}
  for some constant $a=a(n)$ and with $\mu_P$ the multiplicity of $P$
  on $C_{p_i}$. Hence, $\Delta$ must vanish by a similar computation
  as for the low multiplicity points in the proof of
  Proposition~\ref{prop:g}.  Now we are done by taking the union over
  all $p_i$ and all $P$ on $C_{p_i}$ of low multiplicity, similarly as
  for the proof of Proposition~\ref{prop:g}.
\end{proof}
\begin{Rem}\label{rem:highkappa}
Note that, compared to the proof of Theorem~\ref{thm:main}, there are
two further ingredients that lead to higher values of $\kappa$ for
larger $n$, namely, the exponent $2^{n-2}$ of $k$ in Proposition
\ref{lem:alpha:highdim:1}, and, the exponent $1/(n-1)$ of $\mu_P$ in
(\ref{eq:emu:bis}).
\end{Rem}

\begin{proof}[Proof of Theorem~\ref{thm:Pila:highdim;d2}]
  We may suppose that $n>2$ and
  \begin{equation}\label{eq:logHd:Pila}
    1\ll (\log H)^N < d <  H 
  \end{equation}
  for some constant $N=N(n)$ as in
  Proposition~\ref{prop:g:Pila}. 
  Indeed, the case $n=2$ follows from Corollary~\ref{cor:AQ}. If
  $d<(\log H)^N$ then the theorem follows from \cite[Proposition
  2.6]{CDHNV-dgc}, and, if $d>H$ then the theorem follows from
  $\A^n(\Z,H)\le c H^n$ for some constant $c$.  Now the theorem
  follows from Proposition~\ref{prop:g:Pila}, as follows. Let $g_i$ be
  the polynomials of degree $d-1$ provided by
  Proposition~\ref{prop:g:Pila}. Then each intersection $X_i$ of $X$
  with $g_i=0$ is of dimension $n-2$ and of degree at most
  $d(d-1)$. By Schwartz-Zippel for $X_i$ we find
$$
\# X_i(\Q,H) \le c d(d-1)H^{n-2}
$$
for some constant $c$.  By taking the union over the $X_i$ the theorem
follows.
\end{proof}


\section{Proofs 
  for general $K$ }\label{sec:globalproofs}

\subsection{Preliminaries for the global field
  case}\label{sec:prel}

We recall the set-up of Section 2.1 of \cite{Pared-Sas} and Section
4.1 of \cite{CDHNV-dgc}, almost
verbatim. 

We fix a global field $K$, i.e.\ a finite extension of $\Q$ or
${\mathbb{F}}_q(t)$ for some prime power $q$. If $K$ is a finite
extension of ${\mathbb{F}}_q(t)$ we moreover assume that $K$ is
separable over ${\mathbb{F}}_q(t)$ and that ${\mathbb{F}}_q$ is the
full field of constants of $K$.\footnote{Note that any global field of
  positive characteristic is a finite separable extension of a field
  isomorphic to $\FF_q(t)$ for some prime power $q$, by the existence
  of separating transcendental bases.} We denote by $d_K$ either
$[K:\Q]$ or $[K:{\mathbb{F}}_q(t)]$ depending on whether $K$ is an
extension of $\Q$ or of ${\mathbb{F}}_q(t)$.

Let us recall some basics of the theory of heights on $K$, based on \cite{Bomb-Gubl}, where we follow the same
normalizations as in~\cite{CDHNV-dgc} and \cite{Pared-Sas}. Assume
first that $K$ is a number field. Denote by $M_K$ the set of places
of $K$. The infinite places on $K$ come from
embeddings $\sigma: K\to \C$ which give a place $v$ via
\[
  |x|_v = |\sigma(x)|^{n_v/d_K},
\]
where $|\cdot |$ is the usual absolute value on $\R$ or $\C$, and
$n_v = 1$ if $\sigma(K)\subset \R$ and $n_v = 2$
otherwise. 
The finite places of $K$ correspond to non-zero prime ideals of the
ring of integers $\cO_K$ of $K$. For such a prime ideal $p$, we obtain a
place $v$ via
\[
  |x|_v = \cN_K(p)^{-\ord_p(x) / d_K},
\]
where $\cN_K(p) = \# \cO_K / p$ is the norm of $p$.

Now assume that $K$ is a function field. Any place of $K$ then
corresponds to a discrete valuation ring $\cO$ of $K$ containing
${\mathbb{F}}_q$ and whose fraction field is $K$. Let $p$ be the
maximal ideal of $\cO$. Then we define a place $v$
\[
  |x|_v = \cN_K(p)^{-\ord_p(x) / d_K},
\]
where as before $\cN_K(p) = \# \cO/p$. Denote by $M_K$ the set of places
of $K$. We also fix a place $v_\infty$ above the place in
${\mathbb{F}}_q(t)$ defined by $|f|_\infty = q^{\deg f}$.
The ring of integers $\cO_K$ of $K$ is then defined as the set of
$x\in K$ for which $|x|_v\leq 1$ for all places $v\neq v_\infty$.

Let $K$ be an arbitrary global field again (with notation as above). With the above
definitions, recall that we have a product formula, which states that
\[
  \prod_{v\in M_K} |x|_v = 1,
\]
for all $x\in K$. For a point $(x_0 : \ldots : x_n)\in \P^n(K)$ we
define the (absolute projective multiplicative) height to be
\begin{equation}\label{eq:abs:h}
  H(x) = \prod_{v\in M_K} \max_i |x_i|_v.
\end{equation}
This is well-defined because of the product formula. If $x\in K$ put
$H(x) = H(1:x)$.
For an integer $B\geq 1$ let $[B]_{\cO_K}$ be the elements
$x\in \cO_K$ such that
\[
  \begin{cases}
    \max_{\sigma: K\hookrightarrow \C} |\sigma(x) | \leq B^{1/d_K}, & \text{ if $K$ is a number field, and,} \\
    |x|_{v_\infty} \leq B^{1/d_K}, & \text{ if $K$ is a function field},
  \end{cases}
\]
and where the maximum is over embeddings of $K$ in $\C$.  If
$X\subset \A^n_K$ is an affine variety then we write
\begin{equation}\label{eq:cOK}
  X(\cO_K,B)
\end{equation}
for $X(K)\cap [B]_{\cO_K}^n$. These definitions ensure that if
$x\in \P^n(K)$ with $H^{d_K}(x)\leq B$, then there exists a point
$y = (y_1, \ldots, y_{n+1})$ in $\A^{n+1}(K)$ for which
$(y_1 : \ldots : y_{n+1}) = x$ and such that
$y\in [O_K(1) B]^{n+1}_{\cO_K}$ and such that moreover $(y_i\bmod p)$
is well-defined in $\P^n(\cO_K/p)$ for any prime ideal $p$ of a finite
place (namely, the associated place is not in $M_{K, \infty}$) with
$\cN_K(p)>c_2$ for some constant $c_2=c_2(K)$, see
\cite[Prop.\,2.2]{Pared-Sas}.  Hence, for such $p$, $x$ and $y$, we
can define the reduction of $x$ modulo $p$ by $y\bmod p$ in
$\P^n(\cO_K/p)$.

\subsection{Proof of Theorem~\ref{thm:main} for general $K$}

  We explain the adaptations compared to the case $K=\Q$ treated above.

  In Step 1, one notes that the proof of \cite[Proposition
  4.3.2]{CCDN-dgc} goes through in any characteristic, see the
  discussion above Corollary 2.2 of \cite{CDHNV-dgc}.  Instead of
  \cite[Theorem 2]{CCDN-dgc}, one uses \cite[Theorem 1.8]{Pared-Sas}
  (which has a higher power of $d$ when $K$ is of positive
  characteristic, namely, an extra factor $d^4$ compared to
  (\ref{eq:d4})). With these changes to the above Step 1, it is clear that we may assume that $n=2$ and
  \begin{equation}\label{eq:logHd:K}
    1\ll_K (\log H)^2 < d <  H^{3/2}.  
  \end{equation}
  Thus, again by B\'ezout, we may focus on the following analogue of
  Proposition~\ref{prop:g} for general $K$ in order to finish the proof of Theorem~\ref{thm:main} for general $K$.
  \begin{Prop}\label{prop:g:K}
    With any $K$, $C$, $d$, and $H$ as in Theorem~\ref{thm:main},
    $\kappa=12$, and under the extra assumptions $n=2$ and
    (\ref{eq:logHd:K}), there exist no more than $c(\log H)^\kappa$
    many polynomials $g_i$ of degree $d-1$ whose joint zero locus
    contains $C(K,H)$, for some constant $c=c(K)$ depending only on
    $K$.
  \end{Prop}
  This concludes the adaptation of Step
  1. 

Also Step 2 goes through with small natural adaptations, using prime
  ideals instead of prime numbers, and Lemma 3.14 of
  \cite{Pared-Sas} instead of Corollary 2.5 of \cite{CCDN-dgc}.

Let us give the details for this adaptation of step 2 from $K=\QQ$ to general $K$. Constants are allowed to depend on $K$ without notice.
Take any prime ideal $p$ of $\cO_K$ with $\cN_K(p) > \log
H$, and let $C_p$ be the scheme-theoretic reduction of $C$ modulo $p$
inside $\P^2_{\mathbb{F}_p}$, where $\mathbb{F}_p$ stands for the finite field $\cO_K/p$. By \cite[Remark 3.15]{Pared-Sas}, we may assume that $C_p$ is still a curve, of degree at
most $d$. Consider a point $P$ on $C_p(\FF_p)$ and write $\mu_P$ for
$\mult_P(C_p)$, the multiplicity of $P$ on $C_p$. Let $\Xi$ be a
collection of points in $C(\Q,H)$ such that each point in $\Xi$ has
some primitive integer tuple of homogeneous coordinates which reduces
to $P$ modulo $p$, as explained at the end of Section \ref{sec:prel}.  Form an interpolation
determinant $\Delta$, namely, the determinant of a matrix
$(m_i(\xi_j))_{i,j}$ consisting of all monomials $m_i$ of degree $d-1$
in $x_0,x_1,x_2$ evaluated in some points $\xi_j$ of $\Xi$, or more
precisely, evaluated in a tuple
$(\xi_{j1},\xi_{j2},\xi_{j3})$ of homogeneous coordinates in $[O_K(1) H]^{n+1}_{\cO_K}$ of $\xi_j$,
for some $\xi_j$ in $\Xi$, as explained at the end of Section \ref{sec:prel}. This matrix is of dimension $s \times s$
with $s=d(d+1)/2$. 
Hence, from
(2-3) and (2-5) of \cite{Pared-Sas} and (\ref{eq:logHd:K}), we find
$$
\cN_K(\Delta)  \ll (d^2! H^{d^3})^{d_K} \ll d^{d_K2d^2} H^{d_Kd^3}
\ll
H^{cd^3}
$$
for some constant $c$, where $\cN_K(x)$ for nonzero $x$ in $\cO_K$ is $\#\cO_K/(x)$ and $\cN_K(0)=0$.
By Lemma 3.14 of
  \cite{Pared-Sas},
we find that $\Delta$ is divisible by $p^e$ with
\begin{equation}\label{eq:emu:K}
  e \ge \frac{s^2}{2\mu_P} - as
\end{equation}
for some constant $a$.
 Let us assume that
\begin{equation}\label{eq:muP:K}
  \mu_P  
  < \frac{d}
  {\log H},
\end{equation}
so that 
(\ref{eq:logHd:K}) and (\ref{eq:emu:K}) imply
\begin{equation}\label{eq:e:log:K}
  e 
  \ge a'  \log H\cdot  d^3 
\end{equation}
for some constant $a'>0$.
If $\Delta$ is nonzero, we find 
$$
H^{a' (\log(\log H)) d^3} = (\log H)^{a' (\log H) d^3} < \cN_K(p)^{a' (\log H) d^3} \le \cN_K(p)^{e} \le\cN_K(\Delta) \ll H^{cd^3},
$$
which clearly gives a contradiction
with 
(\ref{eq:logHd:K}). We get the following conclusion.  Assume $n=2$ and (\ref{eq:logHd:K}). Let
  $p$ be a prime ideal of $\cO_K$ of norm $\cN_K(p)>\log H$ and let $P$
  be a point on $C_p(\FF_p)$ with $\FF_p:=\cO_K/p$ and $C_p$ the scheme
  theoretic reduction of $C$ modulo $p$. Let $\Xi$ be the points in $C(K,H)$
  whose reduction modulo $p$ equals $P$ (see the end of Section \ref{sec:prel}). Write $\mu_P$ for the
  multiplicity of $P$ on
  $C_p$.  
  Then, $\mu_P < d/ 
  \log H$ implies that there is a homogeneous polynomial $g$ of degree
  $d-1$ vanishing on all points in $\Xi$.

  For the adaptation of the proof of Proposition~\ref{prop:g} to a
  proof of Proposition~\ref{prop:g:K}, one reasons by the Landau prime
  ideal theorem, resp.~the Riemann hypothesis over function fields,
  instead of the prime number theorem when $K$ has characteristic
  zero, resp.~positive characteristic, (see Section 2D of
  \cite{Pared-Sas}), to find sufficiently many primes to make the
  argument work as in the case $K=\Q$.  Note that
  Proposition~\ref{lem:alpha} is already formulated in the correct
  generality above and is used for the curves of the form $C_p$. Let us give the details.

  \begin{proof}[Proof of Proposition~\ref{prop:g:K}]
  Put $d':= \lfloor N \log H
  \rfloor $ with constant $N$ from Proposition~\ref{lem:alpha}.  
  Consider all prime ideals $p_i$ with $\cN_K(p_i)$ between $\log H$ and $M(\log H)^4$ for
  some large constant $M$. 
  Let $\Xi_s$ be the set of points on $C(K,H)$ whose reduction modulo
  $p_i$ has multiplicity at least $d/\log H$ for each of our primes
  $p_i$. Note that we may assume that $C$ is given by a polynomial $f$ over $\cO_K$ and that the reduction of $f$ modulo $p_i$ is
  of degree $D_i\le d$ for some $D_i>0$, by \cite[Remark 3.15]{Pared-Sas}. Any
  interpolation determinant $\Delta_{\Xi_s,d'}$ with monomials of
  degree $d'$ and points from $\Xi_s$ vanishes modulo all the chosen
  prime ideals $p_i$, by Proposition~\ref{lem:alpha}, and where an interpolation determinant is as explained in Step 2.  Hence, if
  $\Delta_{\Xi_s,d'}$ is nonzero we have
  \begin{equation}\label{eq:d'K}
     H^{c'M(\log H)^{3}} = e^{c'M(\log H)^4 } < \prod_{\log H < \cN_K(p_i) < M(\log H)^4 } \cN_K(p_i) \le \cN_K(\Delta_{\Xi_s,d'}),
  \end{equation}
  for some constant $c'$, where $\cN_K(x)$ for nonzero $x$ in $\cO_K$ is $\#\cO_K/(x)$ and where we have used that there are positive
  constants $c_1$ and $c_2$ with
  \begin{equation}\label{eq:TK}
c_1T< \sum_{\cN_K(p)<T} \log \cN_K(p) < c_2T,
  \end{equation}
  for large $T$, where the sum is over primes smaller than $T$, and
  where (\ref{eq:TK}) follows from the Landau prime
  ideal theorem, resp.~the Riemann hypothesis over function fields,
  when $K$ has characteristic  zero, resp.~positive characteristic, see (2-13) of Section 2D of
  \cite{Pared-Sas}.
  On the other hand, using (2-3) and (2-5) of \cite{Pared-Sas} and (\ref{eq:logHd:K}) we find that
  \begin{equation}\label{eq:d''K}
    \cN_K(\Delta_{\Xi_s,d'}) \ll (d'{}^{2}!)^{d_K} H^{d_Kd'{}^3} \ll d'{}^{2d_Kd'{}^{2}} H^{d_Kd'{}^3}    < H^{2d_Kd'{}^3} \le H^{c''\log H ^3} ,
  \end{equation}
for some constant $c''$.
  By (\ref{eq:d'K}), (\ref{eq:d''K}), and by taking $M$ large enough,
  $\Delta_{\Xi_s,d'}$ must vanish as soon $H$ is large enough (which
  holds by (\ref{eq:logHd:K})). 
  Hence, there is a polynomial $g_0$ of degree $d'$ that vanishes on
  $\Xi_s$. Observe that $d'<d$ by (\ref{eq:logHd:K}).

For the complement of $\Xi_s$ in $C(\Q,H)$ we use the above adaptation of Step 2 to general $K$,
  as follows. 
  For each prime ideal $p_i$ with $\log H < \cN_K(p_i) < M(\log H)^4 $, and each
  point $P$ on $C_{p_i}(\mathbb{F}_{p_i})$ of multiplicity at most
  $d/\log H$, there is a polynomial of degree $d-1$ that vanishes on
  the subset of $C(\Q,H)$ consisting of points which reduce to $P$
  modulo $p_i$.
  There clearly are no more than $c(\log H)^4$ many
  such prime ideals $p_i$, for some constant $c$ (by the Landau prime ideal theorem for number fields, resp.~the Riemann hypothesis over function fields, see Section 2D of
  \cite{Pared-Sas}), and for each such $p_i$ there are no more than
  $c'(\log H)^{8}$ many points on $C_{p_i}(\mathbb{F}_{p_i})$ for some $c'$.  Hence,
  in total there are no more than
$c(\log H)^{12}$
many polynomials, as desired for Proposition~\ref{prop:g:K}, for some
constant $c$.
This concludes the proof of Proposition~\ref{prop:g:K} and of Theorem~\ref{thm:main} for general $K$. 
\end{proof}

We obtain this affine variant of Theorem~\ref{thm:main}, with a
similar proof as for Corollary~\ref{cor:AQ}.

\begin{Cor}[Affine curves]\label{cor:AK}
  Let $K$ 
  be as above, and let $n>1$ be given. There is a constant $c=c(K,n)$
  and an absolute constant $\kappa$ such that for any $d>1$, any
  irreducible algebraic curve $C$ of degree $d$ in $\A^n_K$ and any
  $H>2$ one has
  \begin{equation}\label{eq:Sal:log:AO}
    \# C(\cO_K,H) \le c d^{2}H^{\frac{1}{d}}(\log H)^\kappa.
  \end{equation}
\end{Cor}

\subsection{A variant of Theorem~\ref{thm:Pila:highdim;d2} for global
  fields $K$}

In this section we prove the following generalization of Theorem
\ref{thm:Pila:highdim;d2} for any global field $K$.
\begin{Thm}\label{thm:main:Pila:d2K}
  Let $K$ be 
  as above and let $n>1$ be given. There are constants $c=c(K,n)$ and
  $\kappa=\kappa(n)$ such that for any $d\ge 1$, any irreducible
  hypersurface $X\subset \A_K^n$ of degree $d$ and any $H>2$, one has
  \begin{equation}\label{eq:dg:aff}
    \# X(\cO_K,H) \le c d^{2}H^{\dim X-1+1/d}(\log H)^\kappa.
  \end{equation}
\end{Thm}
In order to show Theorem~\ref{thm:main:Pila:d2K}, we show the
following analogue of Proposition~\ref{prop:g:Pila}.
\begin{Prop}\label{prop:g:Pila:K}
  With $K$, $X$, $d$, $n$ and $H$ as in
  Theorem~\ref{thm:main:Pila:d2K} and with the extra assumptions
  that 
  \begin{equation}\label{eq:logHd:Pila:p:2}
    1\ll (\log H)^N < d <  H 
  \end{equation}
  for some constant $N=N(n)\gg_n 1$, there exist no more than
  $c(\log H)^\kappa$ many polynomials $g_i$ of degree $d-1$ whose
  joint zero locus contains $X(\cO_K,H)$, for some constants
  $c=c(K,n)$ and $\kappa=\kappa(n)$.
\end{Prop}
\begin{proof}
One adapts the proof of Proposition~\ref{prop:g:Pila} in the same way that we adapted the proof of Proposition~\ref{prop:g} to a proof of
  Proposition~\ref{prop:g:K}, where one again uses Lemma 3.14 of
  \cite{Pared-Sas} instead of Corollary 2.5 of \cite{CCDN-dgc}.
\end{proof}
\begin{proof}[Proof of Theorem~\ref{thm:main:Pila:d2K}]
One adapts the proof of Theorem~\ref{thm:Pila:highdim;d2} in the same way that we adapted the proof for the case $K=\Q$ to the general case of Theorem~\ref{thm:main}.
\end{proof}


We can now also give the proofs
of 
Theorems~\ref{thm:main:dg} and~\ref{thm:main:dg:aff} for general $K$.

\begin{proof}[Proof of Theorem~\ref{thm:main:dg:aff}]
  If $d\ge 5$ then we may again suppose that $\log H < d$, by
  \cite[Theorem~1.11]{Pared-Sas}. But then
  Theorem~\ref{thm:main:dg:aff} follows directly from
  Theorem~\ref{thm:main:Pila:d2K}.  For the cases $d=4$, and $d=3$ one
  proceeds similarly as for the cases $d=3$ and $d=4$ in
  \cite{BCN-Sal}, using our Corollary~\ref{cor:AK} instead of Theorem
  2 of \cite{BCN-Sal} for the induction base case with $n=3$. For the cases $d=3$ and $d=4$, see also the proofs of Theorems 1.5, 4.1 and Propositions 3.1, 4.19 of \cite{CDHNV-dgc}, where these details are repeated and where an even more general case is proved (that is, with lighter assumptions on $f$ and on $f_d$).
\end{proof}

\begin{proof}[Proof of Theorem~\ref{thm:main:dg}]
  This follows directly from Theorem~\ref{thm:main:dg:aff} by working
  with the affine cone over $X$.
\end{proof}

\bibliographystyle{plain} \bibliography{nrefs}

\end{document}